\documentclass[12pt,reqno]{amsart}

\usepackage{fullpage}
\usepackage{amsmath}
\usepackage{amscd}
\usepackage{verbatim}
\usepackage{amsfonts}
\usepackage{amsthm}
\usepackage{amssymb}
\usepackage{graphicx}
\usepackage{epsfig}
\usepackage[cmtip,all]{xy}
\usepackage{url}
\usepackage{booktabs}
\usepackage{epstopdf}
\usepackage{xcolor}

\newtheorem{theorem}{Theorem}[section]
\newtheorem{lemma}[theorem]{Lemma}

\newtheorem{proposition}[theorem]{Proposition}

\theoremstyle{definition}
\newtheorem{definition}[theorem]{Definition}
\newtheorem{example}[theorem]{Example}

\theoremstyle{remark}
\newtheorem{remark}[theorem]{Remark}
\newtheorem*{remark*}{Remark}

\newcommand{\ra}[1]{\renewcommand{\arraystretch}{#1}}

\numberwithin{equation}{section}
\numberwithin{figure}{section}

\begin{document}

\title{The geometric realization of a normalized set-theoretic Yang-Baxter homology of biquandles}

\author{Xiao Wang}

\author{Seung Yeop Yang}

\email[Xiao Wang]{wangxiaotop@jlu.edu.cn}

\email[Seung Yeop Yang]{seungyeop.yang@knu.ac.kr}

\address{Department of Mathematics, Jilin University, Changchun, 130012, China}

\address{Department of Mathematics, Kyungpook National University, Daegu, 41566, Republic of Korea}

\begin{abstract}
Biracks and biquandles, which are useful for studying the knot theory, are special families of solutions of the set-theoretic Yang-Baxter equation. A homology theory for the set-theoretic Yang-Baxter equation was developed by Carter, Elhamdadi, and Saito in order to construct knot invariants.
In this paper, we construct a normalized (co)homology theory of a set-theoretic solution of the Yang-Baxter equation. We obtain some concrete examples of non-trivial $n$-cocycles for Alexander biquandles. For a biquandle $X,$ its geometric realization $BX$ is discussed, which has the potential to build invariants of links and knotted surfaces. In particular, we demonstrate that the second homotopy group of $BX$ is finitely generated if the biquandle $X$ is finite.
\end{abstract}



\maketitle

\section{Introduction}
The Yang-Baxter equation has played an important role in various fields such as quantum group theory, braided categories, and low-dimensional topology since it was first introduced independently in a study of theoretical physics by Yang \cite{Yan} and statistical mechanics by Baxter \cite{Bax}. In particular, since the discovery of the Jones polynomial \cite{Jon} in 1984, it has been extensively studied in knot theory\footnote{It is known that a certain solution of the Yang-Baxter equation gives rise to the Jones polynomial \cite{Jon, Tur}.}. As a homological approach, Carter, Elhamdadi, and Saito \cite{CES} defined a (co)homology theory for set-theoretic Yang-Baxter operators, from which they provided a method to generate link invariants, and further developments were made by Przytycki \cite{Prz}\footnote{As homology theories of set-theoretic Yang-Baxter operators, they coincide as shown in \cite{PW}.}. Meanwhile, Joyce \cite{Joy} and Matveev \cite{Mat} independently introduced a self-distributive algebraic structure, called a quandle, which satisfies axioms motivated by the Reidemeister moves, and it has been generalized as a biquandle. Quandles and biquandles are solutions of the set-theoretic Yang-Baxter equation, which have been used to define homotopical and homological invariants of knots and links \cite{Nos1, Nos2, SYan, CES, CJKLS, CEGN}. This paper describes the study of a normalized homology theory of a set-theoretic solution of the Yang-Baxter equation and its geometric realization.

\subsection{Preliminary}
Let $k$ be a commutative ring with unity and $X$ be a set. We denote by $V$ the free $k$-module generated by $X.$ Then, a $k$-linear map $R:V \otimes V \rightarrow V \otimes V$ is called a \emph{pre-Yang-Baxter operator} if it satisfies the equation of the following maps $V \otimes V \otimes V \rightarrow V \otimes V \otimes V:$
$$(R \otimes \text{Id}_{V}) \circ (\text{Id}_{V} \otimes R) \circ (R \otimes \text{Id}_{V})=(\text{Id}_{V} \otimes R) \circ (R \otimes \text{Id}_{V}) \circ (\text{Id}_{V} \otimes R).$$
We call a pre-Yang-Baxter operator $R$ a \emph{Yang-Baxter operator} if it is invertible.\\

The classification of the solutions of the Yang-Baxter equation has been actively studied. Following the study by Drinfel$'$d \cite{Dri}, the set-theoretic solutions of the Yang-Baxter equation have been the focus of various studies \cite{EG, ESS, LYZ1, LYZ2, WX}.

\begin{definition}
For a given set $X,$ a function $R:$ $X \times X \to X\times X$ satisfying the following equation (called a \emph{set-theoretic Yang-Baxter equation})
$$(R\times \textrm{Id}_{X})\circ (\textrm{Id}_{X}\times R)\circ (R\times \textrm{Id}_{X})=(\textrm{Id}_{X}\times R)\circ (R\times \textrm{Id}_{X})\circ (\textrm{Id}_{X}\times R)$$
is called a \emph{set-theoretic pre-Yang-Baxter operator} or a \emph{set-theoretic solution of the Yang-Baxter equation}. In addition, if $R$ is invertible, then we call $R$ a \emph{set-theoretic Yang-Baxter operator}.
\end{definition}

Special families known as \emph{biracks} and \emph{biquandles} are strongly related to the knot theory. Their precise definitions are as follows:

\begin{definition}\label{Definition 1.2}
For a given set $X,$ let $R$ be a set-theoretic Yang-Baxter operator denoted by $$R(A_{1},A_{2})=(R_{1}(A_{1},A_{2}),R_{2}(A_{1},A_{2}))=(A_{3},A_{4}),$$ where $A_{i}\in X$ ($i=1,2,3,4$), and $R_{j}: X \times X \rightarrow X$ ($j=1,2$) are binary operations. We consider the following conditions:
\begin{enumerate}
\item For any $A_{1},A_{3}  \in X,$ there exists a unique $A_{2} \in X$ such that $R_{1}(A_{1},A_{2}) = A_{3}.$\\ In this case, $R_{1}$ is \emph{left-invertible}.

\item For any $A_{2},A_{4} \in  X,$ there exists a unique $A_{1} \in X$ such that $R_{2}(A_{1},A_{2}) = A_{4}.$\\ In this case, $R_{2}$ is \emph{right-invertible}.

\item For any $A_{1} \in X,$ there is a unique $A_{2} \in X$ such that $R(A_{1},A_{2})=(A_{1},A_{2}).$

\end{enumerate}
The algebraic structure $(X, R_{1}, R_{2})$ is called a \emph{birack} if it satisfies the conditions $(1)$ and $(2).$ A birack is a \emph{biquandle} if the condition $(3)$ is also satisfied.
\end{definition}

\begin{remark}
The condition $(3)$ in Definition \ref{Definition 1.2} implies that for any $A_{2} \in X,$ there is a unique $A_{1} \in X$ such that $R(A_{1},A_{2})=(A_{1},A_{2}).$ See Remark $3.3$ in \cite{CES}.
\end{remark}

\begin{example}
\begin{enumerate}
\item Let $C_{n}$ be the cyclic rack of order $n,$ i.e., the cyclic group $\mathbb{Z}_{n}$ of order $n$ with the operation $i * j = i+1$ (mod $n$). Then the function $R:X \times X \rightarrow X \times X$ defined by $$R(i,j)=(R_{1}(i,j), R_{2}(i,j))=(j+1, i-1)$$ forms a set-theoretic Yang-Baxter operator. Moreover, $(C_{n}, R_{1}, R_{2})$ is a biquandle, called a \emph{cyclic biquandle}.
\item \cite{CES} Let $k$ be a commutative ring with unity $1$ and with units $s$ and $t$ such that $(1-s)(1-t)=0$.
    Then the function $R:k \times k \rightarrow k \times k$ given by $$R(a,b)=(R_{1}(a,b), R_{2}(a,b))=((1-s)a+sb,ta+(1-t)b)$$ is a set-theoretic Yang-Baxter operator, and $(k, R_{1}, R_{2})$ forms a biquandle, called an \emph{Alexander biquandle}.
    For example, let $k=\mathbb{Z}_{m}$ with units $s$ and $t$ such that $m=|(1-s)(1-t)|,$ then the function $R$ defined as above forms a set-theoretic Yang-Baxter operator and $\mathbb{Z}_{m;s,t}:=(\mathbb{Z}_{m},R_{1},R_{2})$ is a biquandle.
\end{enumerate}
\end{example}

\vspace{0.5cm}
\section{Normalized homology of a set-theoretic solution of the Yang-Baxter equation}\label{section2}

In this section, we study a normalized homology theory for set-theoretic solutions of the Yang-Baxter equation, defined in a similar way as to obtain the quandle homology \cite{CJKLS} from the rack homology \cite{FRS1,FRS2}. We consturct concrete examples of non-trivial $n$-cocycles for the Alexander biquandles $\mathbb{Z}_{m;s,t}.$

First, we review the homology theory for the set-theoretic Yang-Baxter equation based on \cite{CES}.
For a set $X,$ let $R:X \times X \rightarrow X \times X$ be a set-theoretic Yang-Baxter operator on $X.$ For each integer $n>0,$ we define the $n$-chain group $C_{n}^{YB}(X)$ to be the free abelian group generated by the elements of $X^{n}$ and the $n$-boundary homomorphism $\partial_{n}^{YB}: C_{n}^{YB}(X) \rightarrow C_{n-1}^{YB}(X)$ by $\sum\limits_{i=1}^{n}(-1)^{i+1}(d_{i,n}^{l}-d_{i,n}^{r}),$ where the two face maps $d_{i,n}^{l},d_{i,n}^{r}:C_{n}^{YB}(X) \rightarrow C_{n-1}^{YB}(X)$ are given by
$$d_{i,n}^{l} = (R_{2} \times \textrm{Id}_{X}^{\times (n-2)}) \circ (\textrm{Id}_{X} \times R \times \textrm{Id}_{X}^{\times (n-3)}) \circ \cdots \circ (\textrm{Id}_{X}^{\times (i-2)} \times R \times \textrm{Id}_{X}^{\times (n-i)}),$$
$$d_{i,n}^{r} = (\textrm{Id}_{X}^{\times (n-2)} \times R_{1}) \circ (\textrm{Id}_{X}^{\times (n-3)} \times R \times \textrm{Id}_{X}) \circ \cdots \circ (\textrm{Id}_{X}^{\times (i-1)} \times R \times \textrm{Id}_{X}^{\times (n-i-1)}).$$
Then $C_{*}^{YB}(X):=(C_{n}^{YB}(X),\partial_{n}^{YB})$ forms a chain complex, and the yielded homology $H_{*}^{YB}(X)$ is called the \emph{set-theoretic Yang-Baxter homology} of $X.$

Consider the subgroup $C_{n}^{D}(X)$ of $C_{n}^{YB}(X)$ defined by
$$C_{n}^{D}(X)=\text{span}\{(x_{1},\ldots,x_{n})\in C_{n}^{YB}(X)~|~ R(x_{i},x_{i+1})=(x_{i},x_{i+1}) \text{~for some~} i=1,\ldots,n-1\},$$
if $n \geq 2,$ otherwise we let $C_{n}^{D}(X)=0.$

\begin{proposition}
$(C_{n}^{D}(X),\partial_{n}^{YB})$ is a sub-chain complex of $(C_{n}^{YB}(X),\partial_{n}^{YB}).$
\end{proposition}

\begin{proof}
We need to show that $\partial_{n}^{YB}(C_{n}^{D}(X))\subset C_{n-1}^{D}(X)$ for every $n>0.$\\
Let $(x_{1},\ldots,x_{n})\in C_{n}^{D}(X).$ Then there exists $j \in \{1,\ldots,n-1\}$ such that $R(x_{j},x_{j+1})=(x_{j},x_{j+1}),$ and we denote such $x_{j+1}$ by $\overline{x_{j}}.$\\
We first show that the image of $d_{i,n}^{l}$ is contained in $C_{n-1}^{D}(X).$

$(1)$ Clearly $d_{i,n}^{l}(x_{1},\ldots,x_{j},\overline{x_{j}},\ldots,x_{n})\in C_{n-1}^{D}(X)$ if $i<j.$

$(2)$ The terms for $d_{j,n}^{l}$ and $d_{j+1,n}^{l}$ cancel each other because $R(x_{j},\overline{x_{j}})=(x_{j},\overline{x_{j}});$ thus,
\begin{flalign*}
d_{j+1,n}^{l}(x_{1},\ldots,x_{j},\overline{x_{j}},\ldots,x_{n})
&=d_{j,n}^{l}\circ (\textrm{Id}_{X}^{\times (j-1)} \times R \times \textrm{Id}_{X}^{\times (n-j-1)})(x_{1},\ldots,x_{j},\overline{x_{j}},\ldots,x_{n})\\
&=d_{j,n}^{l}(x_{1},\ldots,x_{j},\overline{x_{j}},\ldots,x_{n}),
\end{flalign*}
and there is a sign difference.

$(3)$ When $i>j,$ denote by $d_{i,n}^{l}(x_{1},\ldots,x_{j},\overline{x_{j}},\ldots,x_{n})=(y_{1},\ldots,y_{j},y_{j+1},\ldots,y_{i-1},x_{i+1},\ldots,x_{n}).$ We prove that $y_{j+1}=\overline{y_{j}}$ so that $d_{i,n}^{l}(x_{1},\ldots,x_{j},\overline{x_{j}},\ldots,x_{n}) \in C_{n-1}^{D}(X).$
By the definition of the face map $d_{i,n}^{l},$ we have
$$(R\times \textrm{Id}_{X})\circ (\textrm{Id}_{X}\times R)(x_{j},\overline{x_{j}},z)=(w,y_{j},y_{j+1}),$$
where $z$ is the $(j+2)$th coordinate of $(\textrm{Id}_{X}^{\times (j+1)} \times R \times \textrm{Id}_{X}^{\times (n-j-3)}) \circ \cdots \circ (\textrm{Id}_{X}^{\times (i-2)} \times R \times \textrm{Id}_{X}^{\times (n-i)})(x_{1},\ldots,x_{j},\overline{x_{j}},\ldots,x_{n})$ and $w$ is the $j$th coordinate of $(\textrm{Id}_{X}^{\times (j-1)} \times R \times \textrm{Id}_{X}^{\times (n-j-1)}) \circ \cdots \circ (\textrm{Id}_{X}^{\times (i-2)} \times R \times \textrm{Id}_{X}^{\times (n-i)})(x_{1},\ldots,x_{j},\overline{x_{j}},\ldots,x_{n}).$ Then
\begin{flalign*}
(w,y_{j},y_{j+1})
&=(R\times \textrm{Id}_{X})\circ (\textrm{Id}_{X}\times R)(x_{j},\overline{x_{j}},z)\\
&=(R\times \textrm{Id}_{X})\circ (\textrm{Id}_{X}\times R)\circ  (R\times \textrm{Id}_{X})(x_{j},\overline{x_{j}},z) \textrm{~because $R(x_{j},\overline{x_{j}})=(x_{j},\overline{x_{j}})$}\\
&=(\textrm{Id}_{X}\times R)\circ (R\times \textrm{Id}_{X})\circ  (\textrm{Id}_{X}\times R)(x_{j},\overline{x_{j}},z) \textrm{~by the Yang-Baxter equation}\\
&=(\textrm{Id}_{X}\times R)(w,y_{j},y_{j+1}).
\end{flalign*}
Therefore, $R(y_{j},y_{j+1})=(y_{j},y_{j+1}),$ i.e., $y_{j+1}=\overline{y_{j}}$ as desired.\\
In the same way as above, one can show that the image of $d_{i,n}^{r}$ is also contained in $C_{n-1}^{D}(X).$
\end{proof}

The homology $H_{n}^{D}(X)=H_{n}(C_{*}^{D}(X))$ is called the \emph{degenerate set-theoretic Yang-Baxter homology groups} of $X.$ Consider the quotient chain complex $C_{*}^{NYB}(X):=(C_{n}^{NYB}(X),\partial_{n}^{NYB}),$ where $C_{n}^{NYB}(X)=C_{n}^{YB}(X) \big/ C_{n}^{D}(X),$ and $\partial_{n}^{NYB}$ is the induced homomorphism. For an abelian group $A,$ define the chain and cochain complexes $C_{*}^{NYB}(X;A):=(C_{n}^{NYB}(X;A),\partial_{n}^{NYB})$ and $C^{*}_{NYB}(X;A):=(C^{n}_{NYB}(X;A),\delta^{n}_{NYB}),$ where
$$C_{n}^{NYB}(X;A)= C_{n}^{NYB}(X) \otimes A, ~ \partial_{n}^{NYB} = \partial^{NYB}_{n} \otimes \textrm{Id}_{A},$$
$$C^{n}_{NYB}(X;A)= \text{Hom}(C_{n}^{NYB}(X), A), ~ \delta^{n}_{NYB} = \text{Hom}(\partial^{NYB}_{n}, \textrm{Id}_{A}).$$

\begin{definition}
Let $R$ be a set-theoretic Yang-Baxter operator on $X.$ For a given abelian group $A,$ the homology group and cohomology group
$$H_{n}^{NYB}(X;A)=H_{n}(C_{*}^{NYB}(X;A))=Z_{n}^{NYB}(X;A) \big/ B_{n}^{NYB}(X;A),$$
$$H^{n}_{NYB}(X;A)=H^{n}(C^{*}_{NYB}(X;A))=Z^{n}_{NYB}(X;A) \big/ B^{n}_{NYB}(X;A)$$
are called the $n$th \emph{normalized set-theoretic Yang-Baxter homology group of $X$ with coefficient group $A$} and the $n$th \emph{normalized set-theoretic Yang-Baxter cohomology group of $X$ with coefficient group $A.$}
\end{definition}

\begin{table}
\centering
\caption{Homology of cyclic biquandles}\label{table:comp1}
\ra{1.3}
\begin{tabular}{@{}lllll@{}}
    \toprule
    $n$&$1$&$2$&$3$&$4$\\
    \midrule
    $H_n^{YB}(C_{3})$  &  $\mathbb{Z} \oplus \mathbb{Z}_{3}$  &  $\mathbb{Z}^{3}$  &  $\mathbb{Z}^{9} \oplus \mathbb{Z}_{3}$  &  $\mathbb{Z}^{27}$\\
    $H_n^{D}(C_{3})$  &  $0$  &  $\mathbb{Z}$  &  $\mathbb{Z}^{5}$  &  $\mathbb{Z}^{19}$\\
    $H_n^{NYB}(C_{3})$  &  $\mathbb{Z} \oplus \mathbb{Z}_{3}$  &  $\mathbb{Z}^{2}$  &  $\mathbb{Z}^{4} \oplus \mathbb{Z}_{3}$  &  $\mathbb{Z}^{8}$\\
    &&&&\\
    $H_n^{YB}(C_{5})$  &  $\mathbb{Z} \oplus \mathbb{Z}_{5}$  &  $\mathbb{Z}^{5}$  &  $\mathbb{Z}^{25} \oplus \mathbb{Z}_{5}$  &  $\mathbb{Z}^{125}$\\
    $H_n^{D}(C_{5})$  &  $0$  &  $\mathbb{Z}$  &  $\mathbb{Z}^{9}$  &  $\mathbb{Z}^{61}$\\
    $H_n^{NYB}(C_{5})$  &  $\mathbb{Z} \oplus \mathbb{Z}_{5}$  &  $\mathbb{Z}^{4}$  &  $\mathbb{Z}^{16} \oplus \mathbb{Z}_{5}$  &  $\mathbb{Z}^{64}$\\
    &&&&\\
    $H_n^{YB}(C_{8})$  &  $\mathbb{Z} \oplus \mathbb{Z}_{8}$  &  $\mathbb{Z}^{8}$  &  $\mathbb{Z}^{64} \oplus \mathbb{Z}_{8}$  &  \\
    $H_n^{D}(C_{8})$  &  $0$  &  $\mathbb{Z}$  &  $\mathbb{Z}^{15}$  &  \\
    $H_n^{NYB}(C_{8})$  &  $\mathbb{Z} \oplus \mathbb{Z}_{8}$  &  $\mathbb{Z}^{7}$  &  $\mathbb{Z}^{49} \oplus \mathbb{Z}_{8}$  &  \\
    \bottomrule\\
    &&&&
\end{tabular}
\end{table}

\begin{table}
\centering
\caption{Homology of Alexander biquandles}\label{table:comp2}
\ra{1.3}
\begin{tabular}{@{}llllllll@{}}
    \toprule
    $n$&$1$&$2$&$3$&\vline&$n$&$1$&$2$\\
    \midrule
    $H_n^{YB}(\mathbb{Z}_{8;3,5})$       &$\mathbb{Z}^{2}$   &$\mathbb{Z}^{4} \oplus \mathbb{Z}_{2}^{2}$     &$\mathbb{Z}^{8} \oplus \mathbb{Z}_{2}^{4} \oplus \mathbb{Z}_{8}^{2} $ &\vline&
    $H_n^{YB}(\mathbb{Z}_{9;4,4})$       &$\mathbb{Z}^{3} \oplus \mathbb{Z}_{3}$   &$\mathbb{Z}^{9} \oplus \mathbb{Z}_{3}^{3}$     \\
    $H_n^D(\mathbb{Z}_{8;3,5})$          &$0$                &$\mathbb{Z}^{2}$                               &$\mathbb{Z}^{6} \oplus \mathbb{Z}_{2}^{2}$ &\vline&
    $H_n^D(\mathbb{Z}_{9;4,4})$          &$0$                &$\mathbb{Z}^{3}$                               \\
    $H_n^{NYB}(\mathbb{Z}_{8;3,5})$      &$\mathbb{Z}^{2}$   &$\mathbb{Z}^{2} \oplus \mathbb{Z}_{2}^{2}$     &$\mathbb{Z}^{2} \oplus \mathbb{Z}_{2}^{2} \oplus \mathbb{Z}_{8}^{2} $ &\vline&
    $H_n^{NYB}(\mathbb{Z}_{9;4,4})$      &$\mathbb{Z}^{3} \oplus \mathbb{Z}_{3}$   &$\mathbb{Z}^{6} \oplus \mathbb{Z}_{3}^{3}$     \\
    &&&&\vline&&&\\
    $H_n^{YB}(\mathbb{Z}_{8;5,5})$       &$\mathbb{Z}^{4} \oplus \mathbb{Z}_{2}$   &$\mathbb{Z}^{24} \oplus \mathbb{Z}_{2}^{3}$     &$\mathbb{Z}^{160} \oplus \mathbb{Z}_{2}^{15} \oplus \mathbb{Z}_{4}$&\vline&
    $H_n^{YB}(\mathbb{Z}_{16;13,13})$       &$\mathbb{Z}^{4} \oplus \mathbb{Z}_{4}$   &$\mathbb{Z}^{24} \oplus \mathbb{Z}_{2}^{2} \oplus \mathbb{Z}_{4}^{3}$     \\
    $H_n^D(\mathbb{Z}_{8;5,5})$          &$0$                &$\mathbb{Z}^{4}$                               &$\mathbb{Z}^{44} \oplus \mathbb{Z}_{2}^{4}$ &\vline&
    $H_n^D(\mathbb{Z}_{16;13,13})$          &$0$                &$\mathbb{Z}^{4}$                               \\
    $H_n^{NYB}(\mathbb{Z}_{8;5,5})$      &$\mathbb{Z}^{4} \oplus \mathbb{Z}_{2}$   &$\mathbb{Z}^{20} \oplus \mathbb{Z}_{2}^{3}$     &$\mathbb{Z}^{116} \oplus \mathbb{Z}_{2}^{11} \oplus \mathbb{Z}_{4} $ &\vline&
    $H_n^{NYB}(\mathbb{Z}_{16;13,13})$      &$\mathbb{Z}^{4} \oplus \mathbb{Z}_{4}$   &$\mathbb{Z}^{20} \oplus \mathbb{Z}_{2}^{2} \oplus \mathbb{Z}_{4}^{3}$     \\
    \bottomrule
    &&&&&&&
\end{tabular}
\end{table}

\begin{example}
The following are some computational results for (normalized) set-theoretic Yang-Baxter homology groups:
\begin{enumerate}
  \item Homology groups of some cyclic biquandles and Alexander biquandles are provided in Table \ref{table:comp1} and Table \ref{table:comp2}.
  \item For a given rack (respectively, a given quandle) $(X,*),$ one can obtain the birack (respectively, the biquandle) $\widetilde{X}$ by defining $R_{1}(A_{1},A_{2})=A_{2}$ and $R_{2}(A_{1},A_{2})=A_{1}*A_{2}$ for all $A_{1}, A_{2} \in X.$ In this case, the rack homology (respectively, quandle homology) of $X$ and the set-theoretic Yang-Baxter homology of $\widetilde{X}$ coincide.
  \item For a given shelf $(X,*)$ (Laver tables, for example) one can construct the set-theoretic solution of the Yang-Baxter equation in the same way as in (2) above, and it is neither a birack nor a biquandle, in general.
      Homology groups of some shelves are given in Table \ref{table:comp3}. In Table \ref{table:comp3}, $L_{2}$ denotes the second Laver table and $T_{2}$ and $T_{3}$ are the shelves provided in \cite{CMP}.

\end{enumerate}
\end{example}

\begin{table}
\centering
\caption{Homology of shelves}\label{table:comp3}
\ra{1.3}
\begin{tabular}{@{}lllll@{}}
    \toprule
    $n$&$1$&$2$&$3$&$4$\\
    \midrule
    $H_n^{YB}(L_{2})$  &  $\mathbb{Z}$  &  $\mathbb{Z}$  &  $\mathbb{Z}$  &  $\mathbb{Z}$\\
    $H_n^{NYB}(L_{2})$  &  $\mathbb{Z}$  &  $0$  &  $0$  &  $0$\\
    &&&&\\
    $H_n^{YB}(T_{2})$  &  $\mathbb{Z}^{2}$  &  $\mathbb{Z}^{4}$  &  $\mathbb{Z}^{8}$  &  $\mathbb{Z}^{16}$\\
    $H_n^{NYB}(T_{2})$  &  $\mathbb{Z}^{2}$  &  $\mathbb{Z}^{2}$  &  $\mathbb{Z}^{2}$  &  $\mathbb{Z}^{2}$\\
    &&&&\\
    $H_n^{YB}(T_{3})$  &  $\mathbb{Z}^{3}$  &  $\mathbb{Z}^{9}$  &  $\mathbb{Z}^{27}$  &  $\mathbb{Z}^{81}$\\
    $H_n^{NYB}(T_{3})$  &  $\mathbb{Z}^{3}$  &  $\mathbb{Z}^{6}$  &  $\mathbb{Z}^{12}$  &  $\mathbb{Z}^{24}$\\
    \bottomrule\\
    &&&&
\end{tabular}
\end{table}

It is natural to ask whether the set-theoretic Yang-Baxter homology groups can be split into the normalized and degenerated parts. General degeneracies and decompositions in the set-theoretic Yang-Baxter homology of a semi-strong skew cubical structure have been discussed in \cite{LV}. However, the above is not a semi-strong skew cubical structure in general. It was proven in \cite{PVY} that the set-theoretic Yang-Baxter homology of a cyclic biquandle can be split into the normalized and degenerated parts.

\subsection{The $n$-cocycles of the Alexander biquandles}

We investigate some non-trivial cocycles of the Alexander biquandles, which could later be used to compute the homology groups and classify knots and links.

\begin{lemma} \label{facemaps}
For an Alexander biquandle $X,$ the face maps $d_{i,n}^{l},d_{i,n}^{r}:C_{n}^{YB}(X) \rightarrow C_{n-1}^{YB}(X)$ have the formulas:
$$d_{i,n}^{l}(x_{1},\ldots,x_{n})=(tx_{1}+(1-t)x_{i}, \ldots, tx_{i-1}+(1-t)x_{i}, x_{i+1}, \ldots, x_{n}),$$
$$d_{i,n}^{r}(x_{1},\ldots,x_{n})=(x_{1}, \ldots, x_{i-1}, (1-s)x_{i}+sx_{i+1}, \ldots, (1-s)x_{i}+sx_{n} ).$$
\end{lemma}

\begin{proof}
Since $(1-s)(1-t)=0,$ we have the identities $s(1-t)=1-t$ and $t(1-s)=1-s.$ By direct computation, one can obtain the above formulas.
\end{proof}

\begin{theorem}
Let $X=\mathbb{Z}_{m;s,t}$ be an Alexander biquandle. For $n \geq 2,$ the map $\theta_{n} \in C_{NYB}^{n}(X; \mathbb{Z}_{m})$ defined by
$$\theta_{n}(x_{1},\ldots,x_{n})=\prod_{i=1}^{n-1}(x_{i}-x_{i+1})$$
and extending linearly to all elements of $C^{NYB}_{n}(X)$ is an $n$-cocycle\footnote{For example, $[\theta_{2}]$ is non-trivial when $X=\mathbb{Z}_{8;3,5}$, $\mathbb{Z}_{9;4,7}$, $\mathbb{Z}_{15;11,7}$}.
\end{theorem}

\begin{proof}
(i) Note that $\overline{x}=x$ in $\mathbb{Z}_{m;s,t}$ as $s$ and $t$ are units. Thus, $\theta_{n}(x_{1},\ldots,x_{n})=0$ for every $(x_{1},\ldots,x_{n})\in C_{n}^{D}(X).$ \\
(ii) We next show that $\theta_{n} \circ \partial_{n+1}^{YB}= \theta_{n} \circ \left( \sum\limits_{i=1}^{n+1}(-1)^{i+1}(d_{i,n+1}^{l}-d_{i,n+1}^{r}) \right) = 0$ for each $n \geq 2.$\\
Let $(x_{1},\ldots,x_{n+1}) \in C_{n+1}^{YB}(X).$ By Lemma \ref{facemaps} we have
\begin{flalign*}
\theta_{n} \circ d_{i,n+1}^{l} (x_{1},\ldots,x_{n+1})
&= \theta_{n}  (tx_{1}+(1-t)x_{i}, \ldots, tx_{i-1}+(1-t)x_{i}, x_{i+1}, \ldots, x_{n+1})\\
&= t^{i-1}(x_{1}-x_{2}) \cdots (x_{i-2}-x_{i-1}) (x_{i-1}-x_{i}) (x_{i+1}-x_{i+2}) \cdots (x_{n}-x_{n+1})\\
&+ t^{i-2}(x_{1}-x_{2}) \cdots (x_{i-2}-x_{i-1}) (x_{i}-x_{i+1}) (x_{i+1}-x_{i+2}) \cdots (x_{n}-x_{n+1}).
\end{flalign*}
Therefore,
$\theta_{n} \circ \left( \sum\limits_{i=1}^{n+1}(-1)^{i+1}d_{i,n+1}^{l} \right) (x_{1},\ldots,x_{n+1}) =0.$
Similarly, one can obtain\\ $\theta_{n} \circ \left( \sum\limits_{i=1}^{n+1}(-1)^{i}d_{i,n+1}^{r} \right) (x_{1},\ldots,x_{n+1}) =0,$ as desired.
\end{proof}

\vspace{0.5cm}
\section{Biquandle spaces and their homotopy groups}\label{section3}

A \emph{pre-simplicial set}\footnote{Eilenberg and Zilber \cite{EZ} introduced the notion of a simplicial set under the name of \emph{complete semi-simplicial complex}. A semi-simplicial complex in \cite{EZ} is now called a \emph{pre-simplicial set} \cite{Lod, May}.} $\mathcal{X}=(X_{n},d_{i})$ is a collection of sets $X_{n},$ $n \geq 0$ together with face maps $d_{i}:=d_{i,n}:X_{n} \rightarrow X_{n-1},$ which are defined for $0 \leq i \leq n$ and satisfy the relation
$$d_{i}d_{j}=d_{j-1}d_{i} \hbox{~for~} i < j.$$
The dependencies with respect to $n$ are typically omitted from its notation.
A pre-simplicial set can be turned into a chain complex $(C_{n},\partial_{n}),$ where $C_{n}=\mathbb{Z}X_{n}$ is the free abelian group generated by the elements of $X_{n},$ and $\partial_{n}$ is the linearization of $\sum\limits_{i=0}^{n}(-1)^{i}d_{i}.$\\

The geometric realization $|\mathcal{X}|$ of a pre-simplicial set $\mathcal{X}=(X_{n},d_{i})$ is the cell complex constructed by gluing together standard simplices with the instruction provided by $\mathcal{X}.$ A detailed construction is as follows:
$$|\mathcal{X}|=\coprod\limits_{n \geq 0}(X_{n} \times \triangle^{n}) \Big/ \sim_{s},$$
where
\begin{itemize}
  \item each set $X_{n}$ is endowed with the discrete topology;
  \item $\triangle^{n}=\{(t_{0},\ldots,t_{n}) \in \mathbb{R}^{n+1} \mid \sum\limits_{i=0}^{n}t_{i}=1, t_{i} \geq 0\}$ is the $n$-simplex with its standard topology;
  \item the equivalence relation $\sim_{s}$ is defined by $(\textbf{x}, d^{i}(\textbf{t})) \sim_{s} (d_{i}(\textbf{x}),\textbf{t}),$\\
  where $\textbf{x} \in X_{n},$ $\textbf{t} \in \triangle^{n-1}$ and $d^{i}:=d^{i,n}:\triangle^{n-1} \rightarrow \triangle^{n}$ are the coface maps given by $d^{i}(t_{0}, \ldots, t_{n-1})=(t_{0},\ldots ,t_{i-1},0,t_{i},\ldots,t_{n-1})$ for $0 \leq i \leq n$ and satisfy the relation $d^{i}d^{j-1}=d^{j}d^{i}$ if $0 \leq i < j \leq n.$
\end{itemize}
Note that $|\mathcal{X}|$ is a cell-complex with one $n$-cell for each element of $X_{n},$ and the homology of $|\mathcal{X}|$ coincides with that of the chain complex $(C_{n},\partial_{n}).$

The cubical category can be developed in a similar way to the simplicial one. A \emph{pre-cubical set}\footnote{The cubical category was considered first by J.-P. Serre in his PhD thesis \cite{Ser}. See \cite{FRS2} for details.} $\mathcal{X}=(X_{n},d_{i}^{\varepsilon})$ is a collection of sets $X_{n},$ $n \geq 0$ with face maps $d_{i}^{\varepsilon}:=d_{i,n}^{\varepsilon}:X_{n} \rightarrow X_{n-1},$ which are defined for $1 \leq i \leq n,~ \varepsilon \in \{0,1\}$ and satisfy the relation
$$d_{i}^{\varepsilon}d_{j}^{\delta}=d_{j-1}^{\delta}d_{i}^{\varepsilon} \hbox{~if~} i < j \hbox{~for~} \delta,\varepsilon \in \{0,1\}.$$
A pre-cubical set can also be turned into a chain complex $(C_{n},\partial_{n}),$ where $C_{n}=\mathbb{Z}X_{n}$ is the free abelian group generated by the elements of $X_{n},$ and $\partial_{n}$ is the linearization of $\sum\limits_{i=1}^{n}(-1)^{i}(d_{i}^{0}-d_{i}^{1}).$
One can obtain the geometric realization $|\mathcal{X}|$ of a pre-cubical set $\mathcal{X}=(X_{n},d_{i}^{\varepsilon})$ by gluing together standard cubes with the instruction provided by $\mathcal{X}:$
$$|\mathcal{X}|=\coprod\limits_{n \geq 0}(X_{n} \times \square^{n}) \Big/ \sim_{c},$$
\begin{itemize}
  \item each set $X_{n}$ is endowed with the discrete topology;
  \item $\square^{n}=[0,1]^{n}$ is the $n$-cube with its standard topology;
  \item the equivalence relation $\sim_{c}$ is defined by $(\textbf{x}, d^{i}_{\varepsilon}(\textbf{t})) \sim_{c} (d_{i}^{\varepsilon}(\textbf{x}),\textbf{t}),$\\
  where $\textbf{x} \in X_{n},$ $\textbf{t} \in \square^{n-1}$ and $d^{i}_{\varepsilon}:=d^{i,n}_{\varepsilon}:\square^{n-1} \rightarrow \square^{n}$ are the coface maps defined by $d^{i}_{\varepsilon}(t_{1},\ldots,t_{n-1})=(t_{1},\ldots,t_{i-1},\varepsilon,t_{i},\ldots,t_{n-1})$ for $1 \leq i \leq n, ~ \varepsilon \in \{0,1\}$ and satisfy the relation $d^{i}_{\varepsilon}d^{j-1}_{\delta}=d^{j}_{\delta}d^{i}_{\varepsilon} \hbox{~if~} 1 \leq i < j \leq n, ~ \delta, \varepsilon \in \{0,1\}.$
\end{itemize}
Again, $|\mathcal{X}|$ is a cell-complex with one $n$-cell for each element of $X_{n},$ and the homology of $|\mathcal{X}|$ coincides with that of the chain complex $(C_{n},\partial_{n}).$

\subsection{Biquandle spaces}

For a given set $X,$ let $R=(R_{1}, R_{2}): X \times X \rightarrow X \times X$ be a set-theoretic Yang-Baxter operator. We define the face maps $d_{i}^{r},d_{i}^{l}: X^{n} \rightarrow X^{n-1}$ by
$$d_{i}^{r} = (\textrm{Id}_{X}^{\times (n-2)} \times R_{1}) \circ (\textrm{Id}_{X}^{\times (n-3)} \times R \times \textrm{Id}_{X}) \circ \cdots \circ (\textrm{Id}_{X}^{\times (i-1)} \times R \times \textrm{Id}_{X}^{\times (n-i-1)}),$$
$$d_{i}^{l} = (R_{2} \times \textrm{Id}_{X}^{\times (n-2)}) \circ (\textrm{Id}_{X} \times R \times \textrm{Id}_{X}^{\times (n-3)}) \circ \cdots \circ (\textrm{Id}_{X}^{\times (i-2)} \times R \times \textrm{Id}_{X}^{\times (n-i)}).$$
Then $\mathcal{X}=(X^{n}, d_{i}^{r}, d_{i}^{l})$ forms a pre-cubical set, where $X^{0}$ is a singleton set $\{ \ast \}.$ In this case the homology of its geometric realization $|\mathcal{X}|$ is the homology for the set-theoretic Yang-Baxter equation in \cite{CES} (see Figure \ref{cells}).\\
\begin{figure}[h]
\centerline{{\psfig{figure=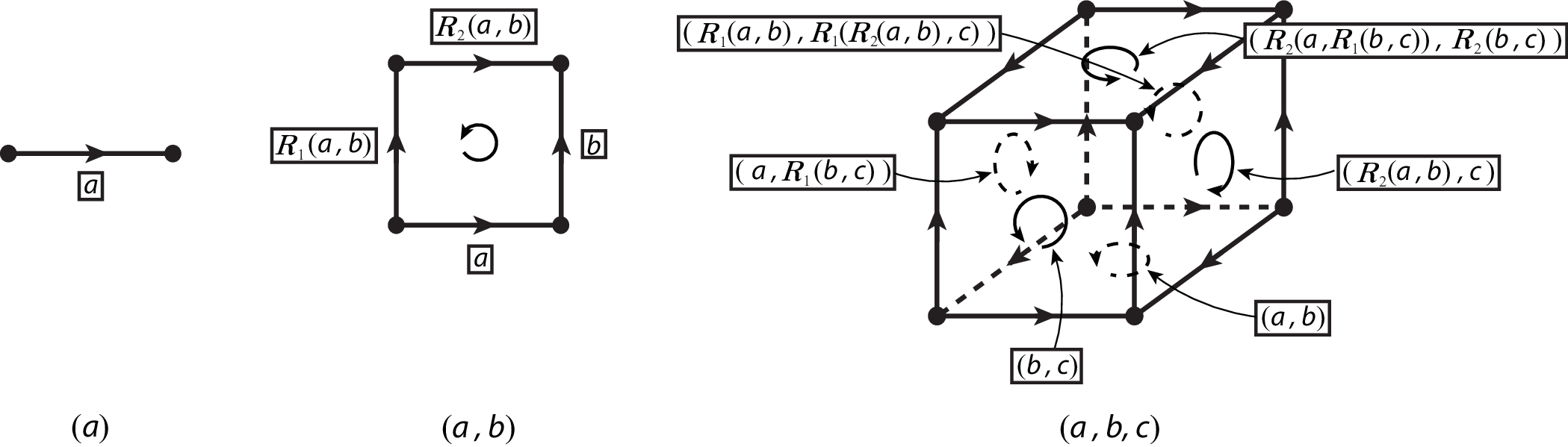,height=4.5cm}}}
\caption{Low-dimensional cells of $|\mathcal{X}|$}
\label{cells}
\end{figure}
When $(X, R_{1}, R_{2})$ is a birack, its geometric realization is called a \emph{birack space}. If $(X, R_{1}, R_{2})$ is a biquandle, the birack space can be transformed into a more interesting space, called a \emph{biquandle space}\footnote{The classifying space of a biquandle was first discussed by Fenn \cite{Fen}.}.
Let $(X, R_{1}, R_{2})$ be a biquandle, and let $|\mathcal{X}|_{n}$ denote the $n$-skeleton of the birack space $|\mathcal{X}|$ of the biquandle $X.$ For each $x \in X,$ we denote the unique element $y \in X$ such that $R(x,y)=(x,y)$ by $\overline{x},$ that is, $R(x,\overline{x})=(x,\overline{x}).$ Let $D^{m}=\{ (x_{1},\ldots ,x_{m}) \in X^{m} ~|~ \overline{x_{i_{0}}} = x_{i_{0}+1}  \hbox{~for some~} 1 \leq i_{0} \leq m-1 \}$ be the subset of $X^{m}.$ Note that $D^{0}=\O=D^{1}.$ In analogy to the way quandle spaces in \cite{Nos1} were obtained from rack spaces \cite{FRS1}, one can construct the $n$-skeleton of the biquandle space $BX_{n}$ by attaching extra cells inductively that bound degenerate cells labeled by the elements of $\bigcup\limits_{m=2}^{n-1}D^{m}$ to $|\mathcal{X}|_{n}.$

The $4$-skeleton $BX_{4}$ of a biquandle space is especially important for classical and surface-knot-theoretic applications, and is given particular attention in this paper. We provide a detailed description of $BX_{4}$ as follows. Let $|\mathcal{X}|_{4}$ be the $4$-skeleton of the birack space of a biquandle $X.$ For each rectangle labeled by $(a, \bar{a}) \in D^{2},$ we glue two edges $\{d_{1}^{\varepsilon}(a, \bar{a})\} \times \square^{1}$ and $\{d_{2}^{\varepsilon}(a, \bar{a})\} \times \square^{1}$ of the rectangle together for every $\varepsilon \in \{l,r\}.$ Then it becomes a $2$-sphere with labeling (see Figure \ref{deg2chain}), and we denote the cone over it by $B_{(a, \bar{a})}^{3}.$ Let $\textbf{x}=(x_{1}, x_{2}, x_{3}) \in D^{3}.$ Then there exists $i_{0} \in \{1,2\}$ so that $\overline{x_{i_{0}}} = x_{i_{0}+1}.$ When we glue two faces $\{d_{i_{0}}^{\varepsilon}(\textbf{x})\} \times \square^{2}$ and $\{d_{i_{0}+1}^{\varepsilon}(\textbf{x})\} \times \square^{2}$ of the cube labeled by $\textbf{x}$ for every $\varepsilon \in \{l,r\},$ it becomes a $S^{2} \times [0,1]$ with labeling, denoted by $S_{\textbf{x}}.$ Then $S_{\textbf{x}} \cup B_{d_{i}^{l}(\textbf{x})}^{3} \cup B_{d_{i}^{r}(\textbf{x})}^{3}$ is a $3$-sphere, where $i \neq i_{0}, i_{0}+1.$ See Figure \ref{deg3chain} for details. The cone over $S_{\textbf{x}} \cup B_{d_{i}^{l}(\textbf{x})}^{3} \cup B_{d_{i}^{r}(\textbf{x})}^{3}$ is denoted by $B_{\textbf{x}}^{4}.$ The space $\big( |\mathcal{X}|_{4} \cup \bigcup\limits_{{\bf x} \in D^{2}} B_{\textbf{x}}^{3} \cup \bigcup\limits_{{\bf x} \in D^{3}} B_{\textbf{x}}^{4} \big) \Big/ \sim_{c}$ is the $4$-skeleton $BX_{4}$ of the biquandle $X.$

\begin{figure}[h]
\centerline{{\psfig{figure=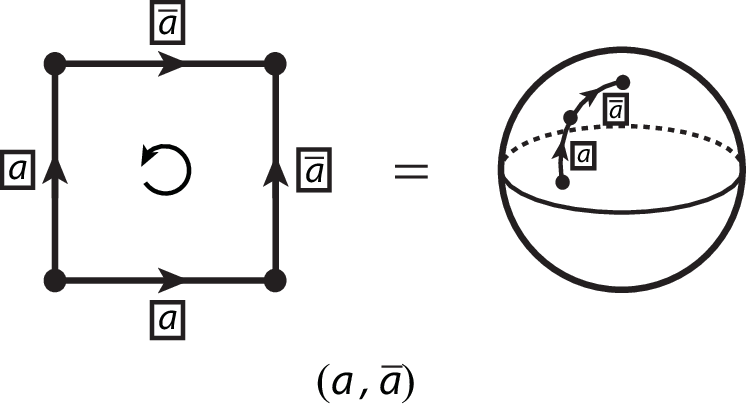,height=3.2cm}}}
\caption{Degenerate $2$-chains}
\label{deg2chain}
\end{figure}

\begin{figure}[h]
\centerline{{\psfig{figure=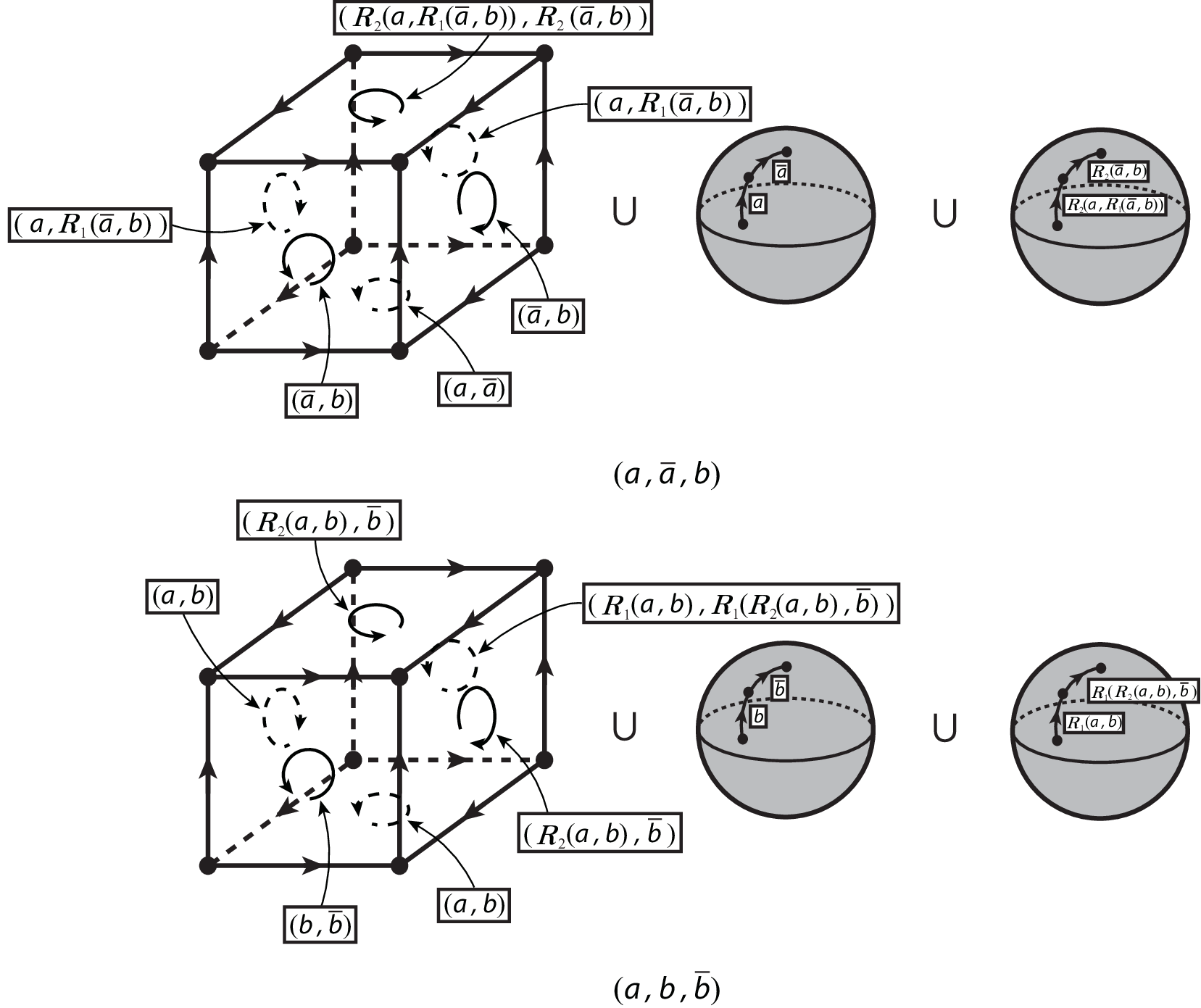,height=8.2cm}}}
\caption{Degenerate $3$-chains}
\label{deg3chain}
\end{figure}

\vspace{0.2cm}
\begin{remark}
Ishikawa and Tanaka \cite{IT} also gave a rigorous definition of a biquandle space, denoted by $B^{Q}X,$ using different face maps and different degeneracies. It is not yet known whether the biquandle space defined above is homotopy equivalent to theirs. However, we do not think they are equivalent in general because in the face map
$$\textbf{x} \mapsto (x_{1}\underline{*}x_{i}, \ldots, x_{i-1}\underline{*}x_{i}, x_{i+1}\overline{*}x_{i}, \ldots, x_{n}\overline{*}x_{i})$$
defined in \cite{IT}, we see that each coordinate is acted by only single element $x_{i}.$

However, for some special biquandles, one can find a continuous map between these two spaces which shows that they are homotopy equivalent to each other. For a cyclic biquandle $X,$ the cellular map $BX \rightarrow B^{Q}X$ assigning each cell labeled by $(x_{1}, x_{2}, \ldots, x_{n})$ to the cell labeled by $(x_{1}, x_{2}+1, \ldots, x_{n}+(n-1))$ is an example.
\end{remark}
\vspace{0.2cm}

\subsection{Homological and homotopical link invariants}

The rack homotopy invariant of framed oriented links, obtained from the classifying spaces of racks, was introduced previously \cite{FRS1}. It was transformed into the quandle homotopy invariant \cite{Nos1, Nos2} and the shadow homotopy invariant \cite{SYan} of oriented links using the classifying spaces of quandles, which are constructed by adding extra cells to the classifying spaces of racks. In a similar manner, we construct a homotopy invariant of oriented links using the classifying spaces of biquandles.

Let $K$ be an oriented link (respectively, an oriented closed knotted surface). Let $(X, R_{1}, R_{2})$ be a biquandle. We call a generic projection of $K$ into $\mathbb{R}^{2}$ (respectively, $\mathbb{R}^{3}$) a \emph{diagram} of $K.$ For each $n=2,3,$ consider the one-point compactification $S^{n}$ of $\mathbb{R}^{n}$ with base point $\infty.$ A biquandle coloring by $X$ of an oriented diagram of $K$ is an assignment of the elements of $X$ to the semi-arcs (respectively, faces) of the oriented diagram with the convention depicted in Figure \ref{biquandlecolor1} (respectively, in Figure \ref{biquandlecolor2}). Note that an $X$-colored crossing (respectively, an $X$-colored triple point) of $K$ represents a chain $\pm(a, b)$ (respectively, $\pm(a, b, c)$) in $C_{n}^{NYB}(X; \mathbb{Z})$ (see Figures \ref{cells}, \ref{biquandlecolor1}, \ref{biquandlecolor2} and compare them to each other). The signs of the chain are determined by the orientations of the corresponding cells. An $X$-colored diagram of $K$ represents a cycle in $Z_{n}^{NYB}(X; \mathbb{Z})$ (for $n=2,3$) that is the signed sum of the chains represented by all crossings of the diagram of $K.$
\begin{figure}[h]
\centerline{{\psfig{figure=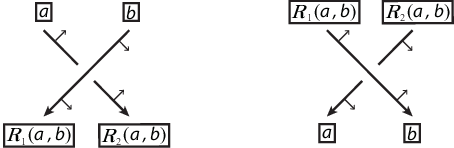,height=3cm}}}
\caption{Coloring convention for oriented links}
\label{biquandlecolor1}
\end{figure}
\begin{figure}[h]
\centerline{{\psfig{figure=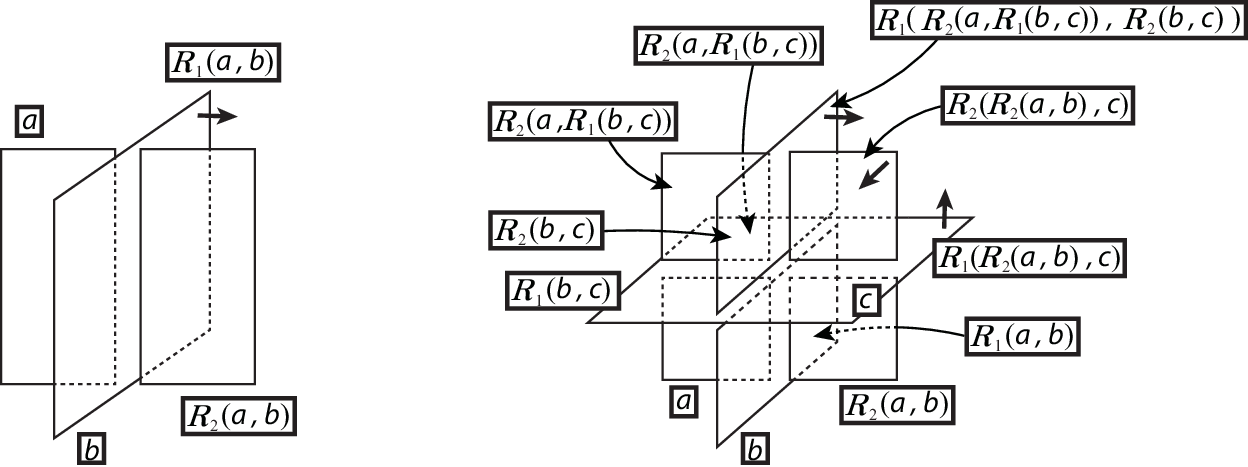,height=5cm}}}
\caption{Coloring convention for oriented knotted surfaces}
\label{biquandlecolor2}
\end{figure}

\begin{proposition}\label{homologicalinv}
The homology class in $H_{n}^{NYB}(X)$ (for $n=2,3$) determined by the represented cycle of a diagram of $K$ is independent of the choice of the diagram.
\end{proposition}

\begin{proof}
Let $D_{1}$ and $D_{2}$ be two diagrams representing $K.$ Note that one diagram can be transformed into the other by a finite sequence of Reidemeister moves or Roseman moves. One can show that each move either does not affect the represented cycle or changes its homology class by a boundary (cf.\cite{CJKLS} Theorem 5.6), i.e., the representative cycles of $D_{1}$ and $D_{2}$ are homologous.
\end{proof}

We let $BX$ be (the $4$-skeleton of) the biquandle space of a given biquandle $X.$ It is well-known that two diagrams represent the same oriented link (respectively, the same closed knotted surface) if and only if they are related by a finite sequence of Reidemeister moves (respectively, Roseman moves). Suppose that the diagram $D$ of an oriented link (respectively, an oriented closed knotted surface) is placed inside $I^{2}$ (respectively, $I^{3}$), where $I=[0,1]$ is the unit interval. By considering $D$ as a decomposition of $I^{2}$ (respectively, $I^{3}$) by an immersed curve (respectively, immersed surface), one can obtain its dual decomposition of $I^{2}$ (respectively, modified dual decomposition of $I^{3}$) (see, e.g., \cite{Nos1,Nos2,SYan} for the definitions). Then each $X$-colored crossing (respectively, each $X$-colored triple point) of $D$ is enveloped in a rectangle (respectively, a rectangular box) labeled by the chain represented by the crossing. The union of the rectangles (respectively, rectangular boxes) is mapped to $BX$ and the boundary of $I^{2}$ (respectively, $I^{3}$) is mapped to the base point $*$ of the biquandle space $BX$ (see Figure \ref{homotopyinvariant} for example). Then the homotopy class of the map $(I^{n}, \partial I^{n}) \rightarrow (BX, \ast)$ is an invariant under Reidemeister moves (respectively, Roseman moves) since every diagrammatic equivalence corresponds to cells in $BX.$

\begin{figure}[h]
\centerline{{\psfig{figure=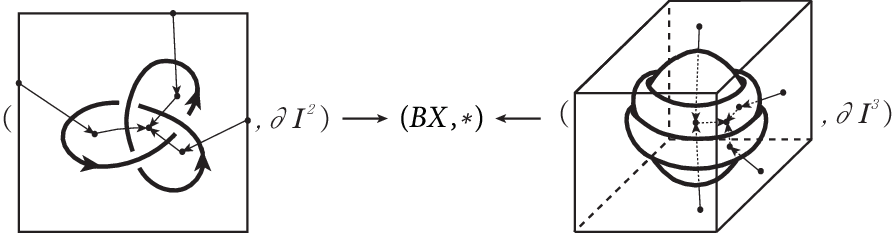,height=3.2cm}}}
\caption{The maps $(I^{n}, \partial I^{n}) \rightarrow (BX, \ast)$ for $n=2,3$}
\label{homotopyinvariant}
\end{figure}

\begin{proposition}(cf.\cite{Nos1,SYan,Nos2})\label{homotopicalinv}
The homotopy class of the maps $(I^{n}, \partial I^{n}) \rightarrow (BX, \ast)$ in $\pi_{n}(BX)$ (for $n=2,3$) is independent of the representative diagram.
\end{proposition}

By using Proposition \ref{homologicalinv} and Proposition \ref{homotopicalinv}, one can construct homological and homotopical link invariants in a similar manner to \cite{CJKLS, Nos1, Nos2}. For example, we define the homotopical state-sum invariants of oriented links and oriented closed knotted surfaces as follows:

Let $K$ be an oriented link or an oriented closed knotted surface, and let $D$ be its diagram. For a given finite biquandle $X,$ we denote by $\text{Col}_{X}(D)$ the set of biquandle colorings of $D$ by $X.$ Then for each $\mathcal{C} \in \text{Col}_{X}(D),$ the homotopy class $\Psi_{X}(D;\mathcal{C})$ of the map $\psi_{X}(D;\mathcal{C}):(S^{n},\infty) \rightarrow (BX,*)$ (for $n=2,3$) is independent of the representative diagram by Proposition \ref{homotopicalinv}. Therefore, $\Psi_{X}(K)=\sum\limits_{\mathcal{C} \in \text{Col}_{X}(D)}\Psi_{X}(D;\mathcal{C}) \in \mathbb{Z}[\pi_{n}(BX)]$ is a link invariant. Here, $n=2$ in the case of an oriented link, and $n=3$ in the case of an oriented closed knotted surface.

\subsection{The second homotopy groups of biquandle spaces}

Some properties of the homotopy groups of quandle spaces were discussed in \cite{Nos1,Nos2}. In a way similar to the idea shown in \cite{FRS1,Nos1,Nos2}, we prove that the second homotopy group of a biquandle space (respectively, a birack space) is finitely generated if the biquandle (respectively, the birack) is finite.



It was shown in \cite{FRS1} that every rack space is a simple space, i.e., $\pi_{1}$ acts trivially on $\pi_{n}$ for each $n > 1.$ We can generalize it on birack spaces as follows.


Let $X$ be a set, and let $R$ be a set-theoretic Yang-Baxter operator on $X.$ The \emph{associated group} or \emph{enveloping group} of $(X,R)$, denoted by $AS(X,R)$ or simply by $AS(X),$ is the group with $X$ as the set of generators and defining relations $xy = R_{1}(x,y)R_{2}(x,y)$ for all $x,y \in X.$ Note that if $X$ is a birack, then $AS(X)$ is isomorphic to $\pi_{1}(|\mathcal{X}|)$ by the definition of the birack space $|\mathcal{X}|.$

\begin{lemma}\label{lemma 3.5}
Let $X$ be a finite birack. Consider the set $\overline{X},$ which consists of the elements in $X$ and their inverse elements in the free group $FX$ generated by $X.$ We denote the symmetric group on $\overline{X}$ by $Sym(\overline{X}).$ Then there exists a homomorphism $\phi:AS(X)\rightarrow Sym(\overline{X})\times Sym(\overline{X})^{op}$ such that
\begin{enumerate}
\item $\emph{Ker}(\phi)$ is finitely generated and abelian;
\item $\emph{Im}(\phi)$  is finite.
\end{enumerate}
\end{lemma}

\begin{proof}
$(1)$ Based on the construction in \cite{LYZ2}, we let $\xi$ and $\eta$ be the left and right actions of $AS(X)$ on itself that extend $R_{1}$ and $R_{2}$ respectively. Consider the homomorphism $\phi:AS(X)\rightarrow Sym(\overline{X})\times Sym(\overline{X})^{op}$ defined by $\phi(a)=(\xi_{\overline{X}}(a),(a)\eta_{\overline{X}})$, where $\xi_{\overline{X}},\eta_{\overline{X}}:AS(X) \rightarrow Sym(\overline{X})$ are the homomorphism induced by $\xi$ and the anti-homomorphism induced by $\eta$ respectively (see \cite{LYZ2} for further details). It is clear that $\textrm{Ker}(\phi)=\textrm{Ker}(\xi_{\overline{X}})\cap \textrm{Ker}(\eta_{\overline{X}}).$ Then Proposition $6$ in \cite{LYZ2} implies that $\textrm{Ker}(\phi)$ is finitely generated and abelian.

$(2)$ Since $X$ is finite and $\textrm{Im}(\phi)$ is a subgroup of $Sym(\overline{X})\times Sym(\overline{X})^{op},$ $\textrm{Im}(\phi)$ is finite.
\end{proof}

\begin{definition}\cite{FRS1}
A \emph{cobordism by moves} between two labelled diagrams is a sequence of the following moves:
\begin{enumerate}
  \item Legal Reidemeister $\Omega_{2}$ and $\Omega_{3}$ moves.
  \item Introduction and deletion of unknotted and unlinked circle components in the diagram $D \Leftrightarrow D \cup \bigcirc.$
  \item A bridge move \begin{figure}[h]{\psfig{figure=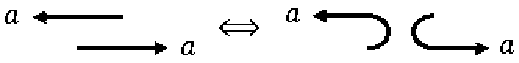,height=0.7cm}}\caption{A bridge move}\end{figure} between adjacent arcs with the same label and opposite orientations.
\end{enumerate}
\end{definition}

\begin{proposition}\cite{FRS1} \label{propFRS}
Homotopy classes of maps $(I^{2}, \partial I^{2}) \rightarrow (|\mathcal{X}|, \ast)$ are in bijective correspondence with equivalence classes of diagrams in $S^{2}$ labelled by $X$ under cobordism by moves.
\end{proposition}

Let us consider the homomorphism $\xi_{\overline{X}}:AS(X) \rightarrow Sym(\overline{X})$ introduced in the proof of Lemma \ref{lemma 3.5}. Let $\widetilde{M}_{X}$ be the connected covering space of $M_{X}:=|\mathcal{X}|$ having $\textrm{Ker} \xi_{\overline{X}}$ as the fundamental group, i.e., $\pi_{1}(\widetilde{M}_{X}) \cong \textrm{Ker} \xi_{\overline{X}} < AS(X) \cong \pi_{1}(|\mathcal{X}|).$ Note that $\widetilde{M}_{X} \rightarrow M_{X}$ is a finite covering because $Sym(\overline{X})$ is of finite order.

\begin{lemma}\label{trivialaction}
For any birack $X,$ the canonical action of $\pi_{1}(\widetilde{M}_{X})$ onto $\pi_{2}(\widetilde{M}_{X})$ is trivial.
\end{lemma}

\begin{proof}
Let $[\gamma] \in \pi_{1}(\widetilde{M}_{X})$ and $[f] \in \pi_{2}(\widetilde{M}_{X}).$ Consider the covering $p:\widetilde{M}_{X} \rightarrow M_{X}$ and its induced homomorphisms $(p_{i})_{*} : \pi_{i}(\widetilde{M}_{X}) \rightarrow \pi_{i}(M_{X})$ ($i=1,2$). Then we have

\begin{flalign*}
(p_{2})_{*}([\gamma] \cdot [f])
&= [p \circ \gamma] \cdot [p \circ f] \text{~} (= (p_{1})_{*}([\gamma]) \cdot (p_{2})_{*}([f])) \text{~} (\text{See Figure~} \ref{trivial action}.)\\
&= [p \circ f] \text{~because~} [p \circ \gamma] \in \textrm{Ker} \xi_{\overline{X}} \text{~} (\text{See Figure~} \ref{action}.)\\
&= (p_{2})_{*}([f]).
\end{flalign*}

\begin{figure}[h]
\centerline{{\psfig{figure=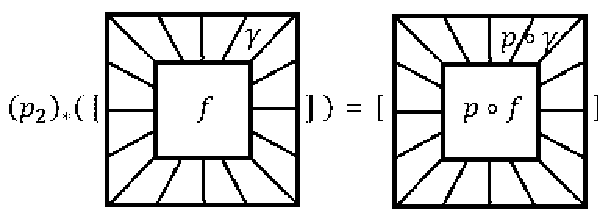,height=2cm}}}
\caption{The canonical action of $\pi_{1}(\widetilde{M}_{X})$ onto $\pi_{2}(\widetilde{M}_{X})$}
\label{trivial action}
\end{figure}

\begin{figure}[h]
\centerline{{\psfig{figure=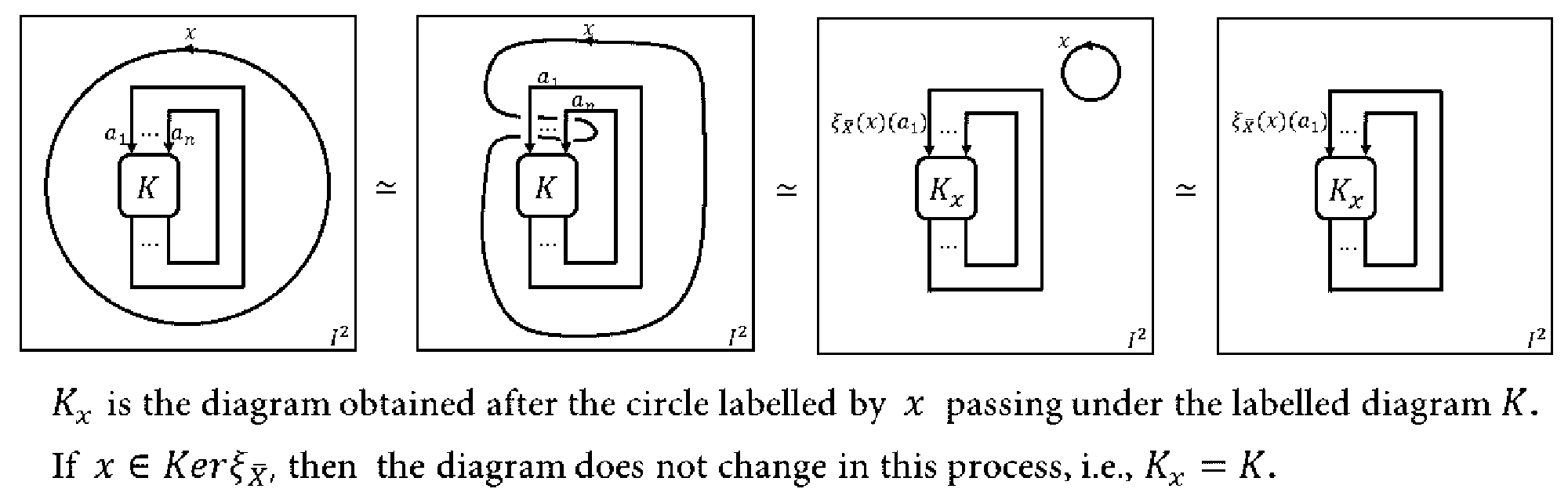,height=4.3cm}}}
\caption{A cobordism by moves between $K$ and $K_{x}$}
\label{action}
\end{figure}

Since $(p_{2})_{*}$ is an isomorphism, $[\gamma] \cdot [f] = [f].$ Therefore, the action is trivial.

\end{proof}

The following proposition gives us a long exact sequence which is used in the proof of Theorem \ref{main}.

\begin{proposition}\cite{McC}\label{exact sequence}
Let $M$ be a connected CW-complex with the trivial canonical action of $\pi_{1}(M)$ on $\pi_{2}(M).$ Then we have the following long exact sequence:
$$\cdots \rightarrow H_{3}(\pi_{1}(M);\mathbb{Z})\xrightarrow{\tau} \pi_{2}(M) \xrightarrow{\mathfrak{h}} H_{2}(M;\mathbb{Z})\rightarrow  H_{2}(\pi_{1}(M);\mathbb{Z})\rightarrow 0,$$
where $H_{3}(\pi_{1}(M);\mathbb{Z})$ is the group homology, $\tau$ is the transgression map, and $\mathfrak{h}$ is the Hurewicz homomorphism.
\end{proposition}

\begin{theorem}\label{main}
For any finite birack $X,$ $\pi_{2}(|\mathcal{X}|)$ is finitely generated. If $X$ is a finite biquandle, then $\pi_{2}(BX)$ is finitely generated.
\end{theorem}

\begin{proof}
$(i)$ Since $X$ is a birack, the sequence
$$H_{3}(\pi_{1}(\widetilde{M}_{X});\mathbb{Z})\xrightarrow{\tau} \pi_{2}(\widetilde{M}_{X}) \xrightarrow{\mathfrak{h}} H_{2}(\widetilde{M}_{X};\mathbb{Z})$$
is exact by Lemma \ref{trivialaction} and Proposition \ref{exact sequence}.

$(ii)$ Since $X$ is finite, its birack space $|\mathcal{X}|$ contains only finitely many $2$-cells, and so does $\widetilde{M}_{X}$ as $\widetilde{M}_{X} \rightarrow M_{X}=|\mathcal{X}|$ is a finite covering. Thus, $H_{2}(\widetilde{M}_{X};\mathbb{Z})$ is finitely generated.

$(iii)$ Let us consider the homomorphism $\phi:AS(X)\rightarrow Sym(\overline{X})\times Sym(\overline{X})^{op}$ defined in Lemma \ref{lemma 3.5} and its restriction $\widetilde{\phi} : \textrm{Ker} \xi_{\overline{X}} \rightarrow Sym(\overline{X})\times Sym(\overline{X})^{op}.$

 Consider the canonical short exact sequence
$$0\rightarrow \textrm{Ker}(\widetilde{\phi})\rightarrow \textrm{Ker} \xi_{\overline{X}}\rightarrow \textrm{Im}(\widetilde{\phi})\rightarrow 0.$$
Then the Lyndon-Hochschild-Serre spectral sequence of the group extension above takes the form
$$E_{p,q}^{2} \cong H_{p}(\textrm{Im}(\widetilde{\phi});H_{q}(\textrm{Ker}(\widetilde{\phi});\mathbb{Z})) \Rightarrow H_{p+q}(\textrm{Ker} \xi_{\overline{X}};\mathbb{Z}).$$
Note that $\textrm{Ker}(\phi)$ is finitely generated and abelian by Lemma \ref{lemma 3.5}$(1).$ Since $\textrm{Ker}(\widetilde{\phi}) < \textrm{Ker}(\phi),$ $\textrm{Ker}(\widetilde{\phi})$ is finitely generated and abelian.
Then $H_{q}(\textrm{Ker}(\widetilde{\phi});\mathbb{Z})$ is finitely generated, and so is $H_{p}(\textrm{Im}(\widetilde{\phi});H_{q}(\textrm{Ker}(\widetilde{\phi});\mathbb{Z}))$ because $\textrm{Im}(\widetilde{\phi}) < \textrm{Im}(\phi)$ is finite by Lemma \ref{lemma 3.5}$(2),$ i.e., $H_{*}(\textrm{Ker} \xi_{\overline{X}};\mathbb{Z})$ is finitely generated. Accordingly, $H_{*}(\pi_{1}(\widetilde{M}_{X});\mathbb{Z})$ is finitely generated since $\textrm{Ker} \xi_{\overline{X}}$ is isomorphic to $\pi_{1}(\widetilde{M}_{X}).$

Therefore, $\pi_{2}(|\mathcal{X}|) = \pi_{2}(M_{X}) \cong \pi_{2}(\widetilde{M}_{X})$ is finitely generated by $(i),$ $(ii),$ and $(iii).$

Moreover, for a biquandle $X,$ the inclusion map $|\mathcal{X}| \hookrightarrow BX$ induces the epimorphism $\pi_{2}(|\mathcal{X}|) \rightarrow \pi_{2}(BX).$ Hence, $\pi_{2}(BX)$ is also finitely generated if $X$ is finite.
\end{proof}

\section*{Acknowledgements}
The authors are grateful to Katsumi Ishikawa and Kokoro Tanaka for valuable conversations on biquandle spaces. The authors would like to thank reviewers for their constructive comments and suggestions, which helped us improve the quality of the paper.\\
The work of Xiao Wang was supported by the National Natural Science Foundation of China (No.11901229).
The work of Seung Yeop Yang was supported by the National Research Foundation of Korea(NRF) grant funded by the Korean government(MSIT) (No. 2019R1C1C1007402 and No. 2022R1A5A1033624).

\end{document}